\documentclass[12pt, twoside, leqno]{article}

\usepackage{amsmath,amsthm}
\usepackage{amssymb}

\usepackage{enumitem}

\usepackage{graphicx}

\usepackage[T1]{fontenc}

\newtheorem{theorem}{Theorem}[section]
\newtheorem{corollary}[theorem]{Corollary}

\newtheorem{proposition}[theorem]{Proposition}

\newtheorem{mainthm}[theorem]{Main Theorem}

\theoremstyle{definition}
\newtheorem{definition}[theorem]{Definition}
\newtheorem{remark}[theorem]{Remark}

\numberwithin{equation}{section}

\DeclareMathOperator{\Lin}{Span}

\begin{document}

\baselineskip=17pt

\title{Most Cantor sets in $\mathbb R^N$ 
are in general position with respect to all projections}

\author{Olga Frolkina\\
E-mail: ofrolkina@gmail.com
}

\date{}

\maketitle


\renewcommand{\thefootnote}{}



\renewcommand{\thefootnote}{\arabic{footnote}}
\setcounter{footnote}{0}

\begin{abstract}
We prove the theorem stated in the title.
This answers a question of John Cobb (1994).
We also consider the case of the Hilbert space~$\ell _2$.
\end{abstract}

\section{Introduction}

Already at the end of the 19th century, mathematicians
were interested in behaviour of projections
of totally disconnected subsets of Euclidean space.
In 1884,  G.~Cantor
described a 
continuous surjection of
what is now called
the standard middle-thirds Cantor set
$\mathcal C$ onto the segment $I=[0,1]$;
it takes a point
$\frac{x_1}{3} + \frac{x_2}{3^2} + \ldots 
\in\mathcal C$ 
(where each
$x_i\in \{0;2\}$)
into the point
$\frac12\left( \frac{x_1}{2}+\frac{x_2}{2^2} + \ldots \right) \in I$.
In 1906, L.~Zoretti
recalls as a known fact
(``ce resultat bien connu etant acquis...'')
that the graph 
of any such surjection
is a Cantor set in plane 
whose projection to $y$-axis coincides with $I$. 
(See \cite{Frolkina2020a} or
\cite{Frolkina2020b} for historic details.)

In 1994, John Cobb asked \cite[p.~126]{Cobb}:
``given integers $m>n>k>0$, is there a Cantor set in $\mathbb R^m$ each of whose projections
into $n$-hyperplanes will be exactly $k$-dimensional?''
Examples of such sets are known for the following triples $(m,n,k)$:
$(2,1,1)$ \cite[{\bf 9}, p.272; fig.2, p.273]{Antoine},
\cite[p.~124, Example]{Cobb},
\cite[Prop.~1]{Dijkstra-van-Mill};
$(m, n, n )$ \cite{Borsuk};
$(3,2,1)$ \cite{Cobb};
$(m,n,n-1)$ \cite{Frolkina2010};
$(m,m-1,k)$ \cite{BDM}.
Other constructions for the cases
$(m,{m-1},{m-2})$ and $(3,2,1)$ are described in 
the papers
\cite{Frolkina2020a} and  \cite{Frolkina2021} which also contain further references.
For each integer $n$, there is a Cantor set in the Hilbert space~$\ell _2$
all of whose projections into $n$-planes are $(n-1)$-dimensional
\cite{BDM}.

In the same paper, Cobb posed another question
\cite[p.~128]{Cobb}:
``Cantor sets that raise dimension under all projections
and those in general position with respect to all projections are 
both dense in the Cantor sets in $\mathbb R^m$ ---
which (if either) is more common, in the sense of
category or dimension or anything?''
In \cite{Frolkina2020b}, I answered a weaker version of this question
showing that
all projections of a typical Cantor set in Euclidean space
are Cantor sets;
thus the examples listed above are ``a rarity''.
In this article, we fully answer the category part of this question:
a typical Cantor set in $\mathbb R^N$ is in general position
with respect to all projections.
We also show that 
a typical Cantor set in $\ell _2$ is in general position
with respect to all projections with finite-dimensional kernels.
(By a projection of  $\mathbb R^N$ or $\ell _2$,
we always mean its orthogonal projection onto a
non-zero linear subspace.)

\section{Preliminaries}

\subsection{Baire Category Theorem}

Let $\mathcal X$ be a non-empty topological space.
We say that a set $\mathcal A\subset\mathcal  X$ is 
\emph{meager in $\mathcal X$}, or
\emph{of first category in $\mathcal X$}, 
if it is the union of countably many
nowhere dense subsets of $\mathcal X$.
A set $\mathcal A$ that is not of first category in $\mathcal X$ is said to be \emph{non-meager in $\mathcal X$}, or
\emph{of second category in $\mathcal X$}.
A set $\mathcal A\subset \mathcal X$ is called 
\emph{comeager in $\mathcal X$},
or 
\emph{residual in $\mathcal X$},
if it is the complement of a meager set;
equivalently, if
it contains the intersection
of a countable family of dense open sets.
The two last definitions make sense only
if $\mathcal X$ is of second category (in itself).

A  topological space $\mathcal X$ is called a \emph{Baire space}
if
any non-empty open set in $\mathcal X$ is non-meager 
in $\mathcal X$;
equivalently, every comeager set is dense in $\mathcal X$;
equivalently, 
for each countable family $\{\mathcal  U_n , n\in\mathbb N \}$ 
of open dense subsets of $\mathcal X$, the intersection
$\bigcap\limits_{n=1}^\infty\mathcal  U_n$
is dense in $\mathcal X$.
Every Baire space is non-meager in itself.
The \emph{Baire Category Theorem} states that
\emph{a completely metrizable space is Baire};
see e.g. \cite[8.4]{Kechris} for a proof.
A $G_\delta $ subset of a completely metrizable
space is itself completely metrizable \cite[3.11]{Kechris}.

Let $\mathcal A$ be a subset of
a non-empty Baire space $\mathcal X$.
\emph{A generic (or typical) element of $\mathcal X$ is in $\mathcal A$}
or 
\emph{most elements of $\mathcal X$ are in $\mathcal A$}
if $\mathcal A$ is comeager in~$\mathcal X$.
This notion is well-defined: in a Baire space,
the statements
``most elements are in $\mathcal A$''
and ``most elements are in $\mathcal X-\mathcal A$''
exclude each other.
In a Baire space $\mathcal X$,
a subset $\mathcal A$ is comeager iff it contains a
dense $G_\delta $ subset of $\mathcal X$.

\subsection{Hyperspaces of compact (Cantor) subsets}

\emph{A Cantor set} is 
a space homeomorphic to 
the standard middle-thirds Cantor set 
$\mathcal C$.
By the Brouwer theorem, these spaces can be characterized as non-empty metric zero-dimensional perfect compacta
\cite[7.4]{Kechris}.

Let $X$ be a complete metric space.
The space 
$\mathcal K (X)$
of all non-empty
compact subsets
of $X$
is endowed with the 
Hausdorff metric;
this space is complete \cite[4.25]{Kechris}.
The corresponding topology on $\mathcal K ( X)$
coincides with the Vietoris topology
\cite[4.F, 4.21]{Kechris}.
If, moreover, $X$ has no isolated points, then
the set $\mathcal C(X)$
of all Cantor sets in $X$
is a dense $G_\delta $ subset of  $\mathcal K (X)$
\cite[Prop. 2]{Kuratowski}
(see also 
\cite[Lemma 1]{Kuzminykh-typ} for the case of $\mathbb R^N$, 
or \cite[Lemma 2.1, Remark 2.2]{Gartside}).
Hence $\mathcal C(X)$ is itself a Baire space.
Note that the Hausdorff metric
on $\mathcal C(X)$ may be not complete.

\section{Main result}

\subsection{Statements}

The next definition 
is introduced in \cite[p. 127]{Cobb} for $\mathbb R^N$.
We extend it to the case of $\ell _2$.

\begin{definition}\label{def:GP}
Let $H$ be a non-zero proper 
linear subspace of $\mathbb R^N$ (or of $\ell _2$).
Let $p_{H}  : \mathbb R^N \to H^\bot $
(or $p_{H}  : \ell _2 \to H^\bot $)
be the orthogonal projection onto $H^\bot $.
For a non-empty subset $M$ of $\mathbb R^N$ (or of $\ell _2$),
denote the fibers of $p_{H} |_M$ by
$F(A) = \left( p_{H} |_M \right) ^{-1} \left( p_{H}  (A) \right)$,
where $A\in M$. Call a fiber $F(A)$ \emph{non-degenerate} if $|F(A)|>1$.
A set $M\subset \mathbb R^N$ (or $M\subset \ell _2$)
is said to be \emph{in general position with respect to
$p_{H} $ or $H$} if, denoting $k=\dim H$, we have:
$p_{H} |_M$ has only finitely many 
non-degenerate fibers;
each non-degenerate fiber $F(A)$ consists of $\leqslant k+1$ points  which form
the vertices of a simplex of dimension
$|F(A)|-1$;
and
$\sum\limits_{A\in M} 
\left( \frac{\left| F( A) \right| - 1 }{|F(A)|}\right) \leqslant k$.
We say that $M$ is \emph{in general position with respect to all projections} if, for \emph{each}~$H$, $M$ is in general position with respect to $H$.
Finally, $M\subset \ell _2$ is \emph{in general position with respect to all projections with finite-dimensional kernels} if $M$ is in general position with respect to $H$ for \emph{each} 
finite-dimensional~$H$.
\end{definition}

\begin{remark}
In \cite[p. 127]{Cobb}, 
the definition of being \emph{in general position with respect to
$p_{H} $ or $H$} looks different.
Namely, the last condition is 
written as\linebreak $\sum\limits_{A\in M} 
( | F( A) | - 1 ) \leqslant k$
instead of 
$\sum\limits_{A\in M} 
\left( \frac{\left| F( A) \right| - 1 }{|F(A)|}\right) \leqslant k$.
But Cobb writes:
``...there are Cantor
sets in $\mathbb R^m$ all of whose projections have 
$0$-dimensional images --- a Cantor set
in a segment is an example; most of its projections are embeddings, while
the others have ``large'' fibers. Of course there cannot be a Cantor set $C$
with every projection an embedding --- for take three points of $C$ and project
parallel to a $2$-plane containing them; or take two pairs of points and project
parallel to a $2$-plane parallel to each pair. We will show that there are Cantor
sets which have at worst the later sort of singularities under projections'' [4, pp. 126-127].
In other words,
the total degeneracy over all distinct fibers should not exceed $k$.
This is exactly what we wrote in Definition~\ref{def:GP}.
(To use Cobb's formula, one should take 
the sum over \emph{distinct} fibers.)
The disagreement between the formula and explanations in \cite{Cobb} was observed by Sergey Melikhov during a talk of mine; I am indebted to him for this remark.
\end{remark}

In $\mathbb R^N$, $N\geqslant 2$,
there exist Cantor sets in general position
with respect to all projections
\cite[Theorem 5]{Cobb} (where a sketch proof is given).
Our main result says
that in the sense of category,
a typical Cantor set in $\mathbb R^N$ 
is in general position with respect to all projections;
this answers the category part of 
Cobb's question \cite[p. 128]{Cobb}.

\begin{mainthm}\label{mainthm}
For each integer $N\geqslant 2$,
Cantor sets in $\mathbb R^N$ 
that are in general position with respect to all projections form
a dense $G_\delta $ subset in the space $\mathcal C(\mathbb R^N)$.
Cantor sets in $\ell _2$ 
that are in general position with respect to all projections
with finite-dimensional kernels form
a dense $G_\delta $ subset in $\mathcal C(\ell _2)$.
\end{mainthm}

Recall that $\mathcal C(X)$ is a
dense $G_\delta $ subset 
in $\mathcal K (X) $ if $X$
is a complete metric space without isolated points.
Hence our main result immediately follows from

\begin{theorem}\label{mainthm-compact}
For each integer $N\geqslant 2$,
non-empty compact sets in $\mathbb R^N$ 
that are in general position with respect to all projections form
a dense $G_\delta $ subset in the space $\mathcal K (\mathbb R^N)$.
Non-empty compacta in $\ell _2$ 
that are in general position with respect to all projections
with finite-dimensional kernels form
a dense $G_\delta $ subset in $\mathcal K(\ell _2)$.
\end{theorem}

Concerning the case of $\ell _2$,
it is appropriate to remind the reader of Theorem~3 from \cite{BDM}
which says that for any integer $k$ and any compact subset
$M\subset \ell _2$,
the restriction
$p_{H} |_M$ is an embedding
for all $H$ in some dense
subset of the Grassmann space
$\mathcal G_k(\ell _2)$.

As a consequence, we get the following result.
Its first part 
was proved in \cite{Frolkina2020b} 
by a straightforward geometric reasoning.

\begin{corollary}\label{cor1}
For each integer $N\geqslant 2$,
all projections of a typical 
Cantor set in $\mathbb R^N$ 
are Cantor sets.
For a 
typical 
Cantor set in $\ell _2$,
any of its projections
with finite-dimensional kernel
is a Cantor set.
\end{corollary}

This corollary follows from Theorem \ref{mainthm} together with Brouwer's characterization of Cantor sets
\cite[7.4]{Kechris}.

\subsection{Partial confirmation}

Before we proceed to the proof of Theorem~\ref{mainthm-compact}, 
let us recall the following standard

\begin{definition}
A subset $M$
of $\mathbb R^N$ (or of $\ell _2$)
is said to be \emph{in general position}
if each of its subsets
of $d+1$ distinct
points $A _0,A_1,\ldots , A _d $, 
where $d\leqslant N$ (respectively, where $d$ is any integer),
is affinely independent;
that is, the affine subspace generated by
$\{ A _0,A _1,\ldots , A _d  \}$ has dimension $d$.
\end{definition}

A Cantor set which satisfies Definition~\ref{def:GP}
is clearly in general position. 
After minor modification,
direct geometric reasoning from
\cite[proof of Lemma~2]{Kuzminykh-typ} give us

\begin{proposition}\label{GP-1}
For any $N\geqslant 2$,
the set of all non-empty compacta
in $\mathbb R^N$ which are
in general position
form a dense $G_{\delta }$ subset in $\mathcal K(\mathbb R^N)$.
The set of all 
non-empty compacta in
$ \ell _2$ which are
in general position
form a dense $G_{\delta }$ subset in $\mathcal K(\ell _2)$.
\end{proposition}

\begin{proof}
Let $X$ be $\mathbb R^N$ (where $N\geqslant 2$) 
or $\ell _2$.
For arbitrary positive integers $d$ and $t$
let $\mathcal H_{d,t}$ be the subset of $\mathcal K(X)$
consisting of all non-empty compacta $K\subset X$
with the property:
there exists an affinely dependent finite subset 
$\{ A_0, A_1, \ldots  , A_d \} \subset K$
such that
 $d\leqslant \dim X$ and
$$\frac{1}{t} \leqslant \min
\{
\rho (A_i , A_j )\ | \ 
i=0,\ldots , d; \ j=0,\ldots , d; \  i\neq j \}
.$$
As in
\cite[proof of Lemma~2]{Kuzminykh-typ}, each $\mathcal H_{d,t}$
is closed and nowhere dense in $\mathcal K(X)$.
A non-empty compactum $M\subset X$
is in general position iff
$M\in \bigcap\limits_{d,t \in \mathbb N} \left(\mathcal K(X)\setminus \mathcal H_{d,t}\right) $.
\end{proof}

But 
to fulfill all conditions
of Definition~\ref{def:GP}, we need additional 
work. 
The main ingredient of our proof
is the 
Mycielski-Kuratowski Theorem.
(Proposition \ref{GP-1} will not be used in our proof,
although it could have been.)

\subsection{Proof of Theorem \ref{mainthm-compact}}

\begin{proof}[Proof of Theorem \ref{mainthm-compact}]
Let $X$ be $\mathbb R^N$ (where $N\geqslant 2$) or $\ell _2$.
Call an array 
of integers 
$(k;s; i_1,\ldots , i_s)$
\emph{admissible} if
 $0< k < \dim X$,
$s\geqslant 1$,
$i_1\geqslant 2$, ...,
$i_s\geqslant 2$, and
$ \sum\limits_{j=1}^s (i_j -1)\geqslant  k+1 $.
For each admissible array,
define the relation
$R_{k;s;i_1,\ldots , i_s}\subset X^{i_1+\ldots +i_s}$
as the set of all 
$$
(A_1^{1},\ldots , A_{i_1}^{1};
A_{1}^{2}, \ldots , A_{i_2}^{2};
\ldots ;
A_{1}^{s}, \ldots , 
A_{i_s}^s ) \in  X^{i_1+\ldots +i_s}
$$
such that the linear span
of the set
\begin{equation}
\label{vectsyst}
\Bigl\{ \overrightarrow{A_1^{j}A_p^{j}}
\ | \ 
 j=1,\ldots ,s, \ \ p=2,\ldots, i_j
 \Bigr\} 
\end{equation}
has dimension $\leqslant k$.
(Note that the set (\ref{vectsyst}) consists of $\sum\limits_{j=1}^s (i_j - 1) \geqslant k+1 $ vectors.)
For an integer $q$ and an ordered $q$-tuple of points
$(A_1,\ldots , A_q)\in X^q$, denote
$$
\mathcal V (A_1, \ldots , A_q):= \{ \overrightarrow{A_1A_2} , 
\overrightarrow{A_1A_3}, \overrightarrow{A_1A_4},
\ldots , \overrightarrow{A_1A_q} \}.
$$
In this notation, the set $(\ref{vectsyst})$
can we written briefly as
$\bigcup \limits_{j=1}^s \mathcal V (A^j_1, A^j_2, \ldots , A^j_{i_j}) $.

For each admissible
array 
$(k;s; i_1,\ldots , i_s)$,
the set 
$R_{k;s;i_1,\ldots , i_s}$ 
is closed and nowhere dense in
$ X^{i_1+\ldots +i_s}$.
Indeed, it is clear that the interior of $R_{k;s;i_1,\ldots , i_s} $ is empty;
and it remains to show that
$ X^{i_1+\ldots +i_s} \setminus R_{k;s;i_1,\ldots , i_s}$ 
is open in
$ X^{i_1+\ldots +i_s}$. For this, just observe that the 
linear subspace spanned by the
system 
$(\ref{vectsyst})$
has dimension $\geqslant k+1$
if and only if 
the set $(\ref{vectsyst})$ contains
a 
subset
$\{ \vec v_1, \ldots ,\vec v_{k+1} \} $ 
with a non-zero Gram determinant
$$
\begin{vmatrix}
(\vec v_1,\vec v_1) & (\vec v_1,\vec v_2) & (\vec v_1,\vec v_3)  & \ldots & (\vec v_1,\vec v_{k+1}) \\
(\vec v_2,\vec v_1) & (\vec v_2,\vec v_2) & (\vec v_2,\vec v_3)  & \ldots & (\vec v_2,\vec v_{k+1}) \\
\vdots & \vdots & \vdots & \ddots & \vdots \\
(\vec v_{k+1},\vec v_1) & (\vec v_{k+1},\vec v_2) & (\vec v_{k+1},\vec  v_3)  & \ldots & (\vec v_{k+1},\vec v_{k+1}) 
\end{vmatrix}
\neq 0.
$$

The Mycielski-Kuratowski Theorem 
\cite[Thm. 1]{Mycielski}, \cite[Cor. 3]{Kuratowski} 
(see \cite[19.1]{Kechris} for a short proof) now implies:
for each admissible array $(k;s; i_1,\ldots , i_s)$
the set
$$
\mathcal M_{k;s;i_1,\ldots , i_s} :=
\{ K \in \mathcal K(X) \ | \ (K)^{i_1+\ldots +i_s} \cap R_{k;s;i_1,\ldots , i_s} = \emptyset \}
$$
is a dense $G_\delta $ subset in $\mathcal K(X)$.
Here 
$
(K)^{q} $ 
denotes the set of all sequences $(A_1,\ldots , A_q) \in K^q$
whose elements are distinct: 
$A_i \neq A_j$ for each $i\neq j$.
Consider
$$
\mathcal P := 
\bigcap 
\mathcal M_{k;s;i_1,\ldots , i_s} 
,
$$
where the 
intersection
is taken over all admissible arrays $(k;s;i_1,\ldots , i_s)$.
The set $\mathcal P$
is a dense $G_\delta $ subset in 
$\mathcal K(X)$.

Claim.
{\sl A non-empty compact set $M\subset X$
belongs to
$\mathcal  P$
if and only if $M$
is in general position with respect to all projections if $X=\mathbb R^N$,
and with respect to all projections with finite-dimensional kernels if $X=\ell _2$.}
Let us prove this.

\underline{Necessity.}
Let $M\in \mathcal  P$.
Take any non-zero proper linear subspace $H$ of $X$
(in addition, assume that $H$ is finite-dimensional if $X=\ell _2$). 
Denote $\kappa := \dim H$.

(i)
Suppose that $p_{H}|_M: M\to H^\bot $
has infinitely many non-degenerate distinct fibers
$F_1, F_2, \ldots $.
Let $\varsigma := \kappa +1 $.
There exist 
integers $\varphi _{1} \geqslant 2$, $\varphi _{2} \geqslant 2$,...,
$\varphi _{\varsigma } \geqslant 2$ and distinct
points
$$
A_1^1, A_2^1 , \ldots , A_{\varphi _{1}  }^1 \in F_1,
$$
$$
A_1^2, A_2^2 , \ldots , A_{\varphi _{2}  }^2 \in F_2,
$$
$$
\ldots 
$$
$$
A_1^{\varsigma }, A_2^{\varsigma } , \ldots , A_{\varphi _{\varsigma }  }^{\varsigma } \in F_{\varsigma} .
$$
Note that
the array
$(\kappa; \varsigma ; \varphi _{1}, \varphi _{2} , \ldots ,
\varphi _{\varsigma })$
is admissible.
By construction, 
$$
(M)^{ \varphi _{1} + \varphi _{2} + \ldots +
\varphi _{\varsigma }}
\subset X^{\varphi _{1} + \varphi _{2} + \ldots +
\varphi _{\varsigma }}
\setminus
R_{\kappa; \varsigma ; \varphi _{1}, \varphi _{2} , \ldots ,
\varphi _{\varsigma }},
$$
hence
$$
\dim \left( \Lin \left(
\bigcup \limits_{j=1}^{\varsigma   }
 \mathcal V (A^j_1, A^j_2, \ldots , A^j_{\varphi _j }) 
\right) \right)
\geqslant \kappa +1 .
$$
On the other hand, for each $j=1,\ldots , \varsigma $
the fiber $F_j$ is contained in the affine subspace 
$A_1^j + H := \{ A_1^j + \vec v \ | \ \vec v\in H\}$,
therefore
$ \mathcal V (A^j_1, A^j_2, \ldots , A^j_{\varphi _j }) 
\subset H$.
We get 
$$
\bigcup \limits_{j=1}^{\varsigma   }
 \mathcal V (A^j_1, A^j_2, \ldots , A^j_{\varphi _j }) 
\subset H ,
$$
and
$$
\dim \left( \Lin \left(
\bigcup \limits_{j=1}^{\varsigma   }
 \mathcal V (A^j_1, A^j_2, \ldots , A^j_{\varphi _j }) 
\right) \right)
\leqslant \dim H = \kappa .
$$
This contradiction shows that
$p_{H}|_M: M\to H^\bot $
has only finitely many non-degenerate fibers.

(ii) Let $F$ be any non-degenerate fiber of $p_{H}|_M: M\to H^\bot $.
Suppose that $|F| \geqslant \kappa +2$;
take distinct points
$A_1,\ldots , A_{\kappa +2}\in F$.
The array $(k;s;i_1) = (\kappa ; 1; \kappa +2)$ is admissible.
Thus $(M)^{\kappa +2}\subset X^{\kappa +2}\setminus R_{\kappa ; 1; \kappa +2}$.
Therefore
$$
\dim \left( \Lin 
(\overrightarrow{A_1A_2},\overrightarrow{A_1A_3}, \overrightarrow{A_1A_4} , \ldots , 
\overrightarrow{A_1A_{\kappa +2}})
\right)
\geqslant \kappa +1.
$$
This leads to 
a contradiction since
$\{
\overrightarrow{A_1A_2},\overrightarrow{A_1A_3}, \overrightarrow{A_1A_4} , \ldots , \overrightarrow{A_1A_{\kappa +2}} \}
\subset
 H
$.
Consequently each non-degenerate fiber of
$p_{H}|_M: M\to H^\bot $
contains $\leqslant \kappa +1$ points.

(iii) Let $F = \{ A_1, \ldots , A_\mu \}$
 be any non-degenerate fiber of $p_{H}|_M: M\to H^\bot $.
Let us show that its points form the vertices of
a $(\mu -1)$-dimensional simplex.
Since $\mu \leqslant \kappa +1$ by (ii),
it suffices to show that 
the system $ \{ A_1, \ldots , A_\mu \}$ is in general position.
Assume the contrary;
then there exists a linear subspace $\tilde H \subset H$
and an index $j _0 \in \{1,\ldots , \mu \}$
such that
$$
|(A_{j _0} + \tilde H) \cap F | \geqslant  \dim \tilde H + 2  .
$$
On the other hand,
the set 
$(A_{j_0} + \tilde H) \cap F$ is a non-degenerate fiber
of the projection $p_{ \tilde H } |_ M : M \to \tilde H^\bot $.
Applying (ii) to $p_{ \tilde H } |_ M $, we get the inequality
$$|(A_{j_0} + \tilde H) \cap F | \leqslant  \dim \tilde H + 1,$$
which is a contradiction.
Therefore 
the system $ \{ A_1, \ldots , A_\mu \}$ is in general position.

(iv) Let $F_1,F_2,\ldots , F_{\varsigma }$ be the list of 
all non-degenerate fibers of 
 $p_{H}|_M: M\to H^\bot $.
They are finite sets by (ii).
Enumerate their points:
 $$
 F_1 = \{ A_1^1, A_2^1, A_3^1,\ldots , A_{\varphi _1}^1 \} ,
 $$
 $$
 F_2 = \{ A_1^2, A_2^2, A_3^2,\ldots , A_{\varphi _2}^2 \} ,
 $$
 $$
 \ldots 
 $$
  $$
 F_{\varsigma } = \{ A_1^{\varsigma }, A_2^{\varsigma }, A_3^{\varsigma },\ldots , A_{\varphi _{\varsigma }}^{\varsigma } \} .
 $$
Our aim is to prove the inequality
 $$
 (\varphi _1 -1) + (\varphi _2 - 1)+ \ldots + (\varphi _{\varsigma } - 1) \leqslant \kappa .
 $$
 Suppose the contrary: let
 $$
 (\varphi _1 -1) + (\varphi _2 - 1)+ \ldots + (\varphi _{\varsigma } - 1) \geqslant \kappa +1.
 $$
In this case, the array $(\kappa ; \varsigma ; \varphi _1, \ldots , \varphi _{\varsigma })$ is admissible. We get
$$(M)^{\varphi _1 + \varphi _2+  \ldots + \varphi _{\varsigma }}
 \subset X^{\varphi _1 + \varphi _2+  \ldots + \varphi _{\varsigma }} \setminus R_{\kappa ; \varsigma ; 
 \varphi _1 , \ldots , \varphi _{\varsigma }},$$
 hence
 $$
 \dim \left( \Lin \left(
\bigcup \limits_{j=1}^{\varsigma   }
 \mathcal V (A^j_1, A^j_2, \ldots , A^j_{\varphi _j }) 
\right) \right)
 \geqslant \kappa +1.
 $$
 This leads to  a contradiction
 since 
$\bigcup\limits_{j=1}^\varsigma  \mathcal V (P^j_1, P^j_2, \ldots , P^j_{\varphi _j }) 
\subset H$, as before.

\underline{Sufficiency.}
Let us prove that 
for $N\geqslant 2$
any non-empty compactum $M\subset \mathbb R^N$ 
(or $M\subset \ell _2$)
which is in general position with respect to all projections
(in case of $\ell _2$, to all projections with finite-dimensional kernels)
lies in $\mathcal P$.
Suppose, towards a contradiction, that
$M\notin \mathcal M_{k;s; i_1,\ldots , i_s}$
for some admissible array $(k;s; i_1,\ldots , i_s)$.
That is,
$(M)^{i_1+\ldots +i_s} \cap R_{k;s;i_1,\ldots , i_s} \neq\emptyset $.
By definition, 
$(M)^{i_1+\ldots +i_s}$ consists
of all ordered
$(i_1+\ldots +i_s)$-tuples of distinct points of $M$.
Then 
$ R_{k;s;i_1,\ldots , i_s} $
contains 
an array
$$
(A_1^{1},\ldots , A_{i_1}^{1};
A_{1}^{2}, \ldots , A_{i_2}^{2};
\ldots ;
A_{1}^{s}, \ldots , 
A_{i_s}^s ) 
$$
consisting
of distinct points of $M$.
We have
$$
\dim \left( \Lin \left(
\bigcup \limits_{j=1}^s \mathcal V (A^j_1, A^j_2, \ldots , A^j_{i_j}) 
\right)
\right) \leqslant k .
$$
Put $H:= \Lin \left(
\bigcup \limits_{j=1}^s \mathcal V (A^j_1, A^j_2, \ldots , A^j_{i_j}) 
\right)$.
Consider 
the fibers of 
 $p_H|_M: M\to H^\bot $.
For each $j=1,\ldots , s$ the set
$\{ A^j_1, A^j_2, \ldots , A^j_{i_j} \}$
is contained in one of the fibers
(these fibers may coincide for different $j$'s).
Consequently 
$$
\sum\limits_{A\in M} 
\left( \frac{\left| F( A) \right| - 1 }{|F(A)|}\right)
\geqslant \sum\limits_{j=1}^s (i_j - 1) \geqslant k+1.
$$
Hence $M$ is not in general position with respect to $p_H$,
a contradiction.
This finishes the proof.
\end{proof}

\end{document}